\documentclass{article}
\usepackage[usenames,dvipsnames]{color}
\usepackage{amsfonts,amsmath,amsthm,amssymb,bbm}
\usepackage[all,dvips]{xy}
\usepackage[dvips]{hyperref}
\newtheorem{theorem}{Theorem}[section]
\newtheorem{lemma}[theorem]{Lemma}
\newtheorem{definition}[theorem]{Definition}

\newtheorem{notation}[theorem]{Notation}

\newtheorem{remark}[theorem]{Remark}
\newtheorem{proposition}[theorem]{Proposition}
\newtheorem{corollary}[theorem]{Corollary}
\begin{document}
\title{The Category of Locales is Rigid}
\author{John Iskra, PhD.}
\maketitle
\begin{abstract}
In this paper we show that the category of frames, and, thus, the category of locales is 'rigid'. This means that every endo-equivalence on them is isomorphic to the identity functor.  To reach this result we prove new results concerning the number of automorphisms between frames and new results concerning the order preserving properties of endo-equivalences.
\end{abstract}
\section{Introduction}
An endo-equivalence on a category is an equivalence (functor) from the category to itself. Following terminology introduced in a question asked \href{http://mathoverflow.net/questions/56887/rigidity-of-the-category-of-schemes}{here}, on mathoverflow.net, I define a rigid category to be one in which all endo-equivlances on it are isomorphic to the identity. In this article I plan to show that the category of locales is rigid.  Taken together, the papers \cite{clark1973automorphism} and \cite{herrlich1991automorphisms} and Peter Freyd's work in \cite{freyd1966abelian}  give a good survey of similar results for other categories. The characterization of Abelian categories by Grothendieck in \cite{grothendieck1957quelques} may be the first well-known result along these lines. In these papers and in other sources, for example, \cite{Borceux}, the belief is expressed that a category which is rigid can be characterized by its \textit{categorical} properties only.  Unfortunately, there does not seem to be wide-spread agreement about just what is meant by a categorical property.  I do not attempt to answer such a deep question here, although I believe its consideration to be greatly important.  What I do is prove that the category or frames is rigid, and, since the category of locales is defined to be the opposite of the category of frames, and since a category is rigid if its dual is, the category of locales is shown to be rigid.\\
   Good sources for finding facts, and definitions concerning frames and locales are \cite{johnstone1986stone}, \cite{Johnstone}, and \cite{Borceux}. Here are the basic definitions, and notations, though. The category of frames consists of objects which are complete lattices in which the infinite distributive law
$$a\wedge\left(\bigvee_{i\in I}b_{i}\right) = \bigvee_{i\in I}\left(a\wedge b_{i}\right)$$
holds. Because frames have all joins and finite meets, they also possess both a largest, or 'top' element and a smallest, or 'bottom' element.  These will be denoted $1$ and $0$ respectively.  Following convention, I use $a\wedge b$ to stand in for the greatest lower bound of the set $\lbrace a,b\rbrace$.  It is also called the meet of $a$ and $b$. Similarly, the symbol $\bigvee$ will stand for the least upper bound and is usually called the join.  Arrows in the category of frames are functions which preserve meets of finite sets and joins of arbitrary ones. The category of locales is usually defined to be the opposite of the category of frames, although, Borceux (\cite{Borceux}) takes a slightly different approach.  In fact this paper was motivated by my desire to understand what, if any, consequences the slightly different definitions had.  As it turns out, because the category of frames (and, thus, the category of locales) is rigid, there are no consequences. In any case, the objects of the two categories are the same, and, so, whether I refer below to an object as a locale of frame, the meaning is the same.\\
 In this paper I depart a bit from convention in that I denote the set of arrows with source $A$ and target $B$ by $\mbox{Arr}\langle A,B\rangle$.  I will also sometimes describe relations between elements in partially ordered sets (including locales) by use of diagrams in which the relation $a\leq b$ is depicted by $a\to b$, or, where $b$ lies above $a$ in the diagram, as in:\\
\centerline{
\xymatrix{
b\\
a\ar@{-}[u]
}
}
Finally, at the end of this next section, I will outline of my strategy to prove the main result. Placing it there will be more helpful to the reader unfamiliar with certain aspects of locale theory which I cover in the section.
\section{The Sierpinski Locale}
We first make use of the idea of a generator in a category.
\begin{definition}
Let $\mathcal{C}$ be a category.  An object $G$ is a generator of $\mathcal{C}$ if for any two unequal arrows\\
\centerline{
\xymatrix{
A\ar@<.5ex>[r]^{f}\ar@<-.5ex>[r]_{g} & B
}
}
there exists an arrow $n$ so that the diagram\\ 
\centerline{
\xymatrix{
G\ar@*{[red]}[r]^{n} & A\ar@<.5ex>[r]^{f}\ar@<-.5ex>[r]_{g} & B
}
}
does \textbf{not} commute; that is, $f\circ n\ne g\circ n$.
\end{definition}
Categories may or may not possess such objects.  The category of frames is one which does.  It is often referred to as the Sierpinski Locale.  As it will be in this paper. 
\begin{definition}
 The Sierpinski locale is the locale $S$ defined by the diagram\\
\centerline{
\xymatrix{
1\\
a\ar[u]\\
0\ar[u]
}
}
\end{definition}
\begin{lemma}\label{Sisgen}
The Sierpinski locale is a generator in the category of frames.
\end{lemma}
\begin{proof}
Let\\
\centerline{
\xymatrix{
L\ar@<.5ex>[r]^{f}\ar@<-.5ex>[r]_{g} & M
}
}
be unequal arrows in the category of frames.  Then, we can find some $x\in L$ so that $f(x)\ne g(x)$.  By definition $x\in L\setminus\lbrace 0,1\rbrace$, that is, $x$ is neither the top nor bottom element of $L$ since any arrow in the category of frames must preserve those elements.  Define $n:S\to L$ by  
\begin{equation}
n(z) = \left\{ \begin{array}{cc}
               1 & \mbox{for $z = 1$}\\
               x & \mbox{for $z = a$}\\
               0 & \mbox{for $x = 0$}
               \end{array}
\right.
\end{equation}
Then, $n$ preserves all joins and meets, and, by construction, $f\circ n\ne g\circ n$.
\end{proof}
As promised in the introduction, we will show that, in addition to the special property of being a generator, the Sierpinski locale is unusual (perhaps unique) in that it can be characterized by the cardinality of its set of automorphisms.  
\begin{lemma}\label{Shas3autos}
There are just three elements in the set $Arr\langle S,S\rangle$ of arrows from $S$ to $S$.
\end{lemma}
\begin{proof}
Any arrow in the category of Frames must take $0$ to $0$ and $1$ to $1$.  Thus, we have three elements in $Arr\langle S,S\rangle$, corresponding to the image of $a$.
\end{proof}
\begin{notation}\label{nota1}
We will call the element of $Arr\langle S,S\rangle$ which sends $a\mapsto 0$, $g$ and the one which sends $a\mapsto 1$, $f$. We will denote the identity arrow on $S$, $1_{S}$.  Similarly, the identity arrow of any object $A$ in a category will be denoted $1_{A}$.
\end{notation}
\begin{notation}
In what follows, the locale $\lbrace 0.1\rbrace$ will be denoted $T$.
\end{notation}
\begin{remark}
In the category of (frames) locales, $T$ is a(n initial) terminal object (p.19 of \cite{Borceux}) 
\end{remark}
\begin{definition}\label{defpts}
Let $L$ be a locale.  Then an arrow $p:L\to T$ in the category of locales is called a point.
\end{definition}
\begin{lemma}\label{Shas2pts}
The locale $S$ has exactly $2$ points.
\end{lemma}
\begin{proof}
  The two functions described in \ref{nota1} are distinct points.  Since any arrow in the category of frames must take $0$ to $0$ and $1$ to $1$, they are the only two.
\end{proof}
\begin{theorem}\label{Sis2pt3autogen}
$S$ is a two-pointed generator which posesses exactly $3$ automorphisms.
\end{theorem}
\begin{proof}
Apply lemmas \ref{Sisgen}, \ref{Shas3autos} and \ref{Shas2pts}.
\end{proof}
\begin{lemma}\label{agoesanywhere}
Let $L$ be a locale and $l$ any element of $L$.  Then, the function $f:S\to L$ defined by
\begin{equation}
f(x) = \left\{ \begin{array}{cc}
               1 & \mbox{for $x = 1$}\\
               l & \mbox{for $x = a$}\\
               0 & \mbox{for $x = 0$}
               \end{array}
\right.
\end{equation}
is a frame arrow.
\end{lemma}
\begin{proof}
We check to see that $f$ preserves joins and meets:\\
Meets:
\begin{enumerate}
\item
$$f(1\wedge a) = f(a) = l = 1\wedge l = f(1)\wedge f(a)$$
\item
$$f(1\wedge 0) = f(0) = 0 = 1\wedge 0 = f(1)\wedge f(0)$$
\item
$$f(a\wedge 0) = f(0) = 0 = 1\wedge 0 = f(a)\wedge f(0)$$
\end{enumerate}
Joins:
\begin{enumerate}
\item
$$f(1\vee a) = f(1) = 1 = 1\vee 1 = f(1)\vee f(a)$$
\item
$$f(1\vee 0) = f(1) = 1 = 1\vee 0 = f(1)\vee f(0)$$
\item
$$f(a\vee 0) = f(a) = l = l\vee 0 = f(a)\vee f(0)$$
\end{enumerate}
\end{proof}
\begin{lemma}\label{partordisomimplisframeisom}
Let $X$ and $Y$ be frames and $f:X\to Y$ a bijection which preserves order. Suppose that the inverse of $f$ also preserves order.  Then, $f$ is an isomorphism in the category of frames.
\end{lemma}
\begin{proof}
Let $g$ be the inverse function of $f$.  Let $S\subseteq X$ be arbitrary.  Let
$$m = \bigvee S$$
Let $s\in S$ be arbitrary.  Then, since $s\leq m$ and $f$ preserves order, $f(s)\leq f(m)$. Since $s$ was arbitrary, 
$$f(m)\geq\bigvee_{s\in S}f(s)$$
Suppose that 
$$t\geq\bigvee_{s\in S}f(s)$$
Let $s\in S$ be arbitrary.  Then, $t\geq f(s)$, thus, $g(t)\geq s$.  Since $s$ was arbitrary, 
$$g(t)\geq\bigvee_{s\in S}(s)$$
Thus,
$$t\geq f\left(\bigvee_{s\in S}(s)\right) = m$$
Thus, $f$ preserves arbitrary joins.\\
Now let $x$ and $w$ be arbitrary elements of $X$. Let $l = x\wedge w$.  Then, $l\leq x$ and $l\leq w$ so $f(l)\leq f(x)$ and $f(l)\leq f(w)$. Thus, $f(l)$ is a lower bound for $\lbrace f(x),f(w)\rbrace$.  Suppose $d$ is a lower bound of $\lbrace f(x),f(w)\rbrace$. Then, $g(d)$ is a lower bound for $\lbrace x,w\rbrace$, thus, $g(d)\leq l$ and so $d\leq f(l)$.  Thus, $f(l) = f(x\wedge w) = f(x)\wedge f(w)$. Thus, $f$ preserves finite meets.\\
Of course, the same argument applies to the inverse of $f$, thus $g$ is also an arrow in the category of frames.  Thus, $f$ is an isomorphism of frames.
\end{proof}
Finally, the Sierpinski locale is important to us because it allows us to recapture the structure of a locale $L$ by examining the arrows $S\to L$.  In particular, we have
\begin{theorem}
Let $L$ be a frame.  Define a partial order on $\mbox{Arr}\langle S,L\rangle$ by 
$$p\leq q\Leftrightarrow p(a)\leq q(a)$$
where, as before, $a\in S$ is neither the top nor bottom element. Then, 
$L$ is isomorphic to $\mbox{Arr}\langle S,L\rangle$ as frames. 
\end{theorem}
\begin{proof}
Define $e:\mbox{Arr}\langle S,L\rangle\to L$ by the assignment $p\mapsto p(a)$. We first check to see that $\mbox{Arr}\langle S,L\rangle$ is a frame. Let $W\subseteq\mbox{Arr}\langle S,L\rangle$. Let $m:S\to L$ be defined by the assignment 
$$a\mapsto \bigvee{p\in W} e(p)$$
Let $p\in W$ be arbitrary.  Then, $p(a)\leq\vee{p\in W} e\circ p$.  Suppose $d$ is an upper bound for $W$. Let $p\in W$ be arbitrary.  Then, by definition of the partial order on $\mbox{Arr}\langle S,L\rangle$, $p(a)\leq d(a)$.  This is identical to saying that $e(p)\leq e(d)$.  Thus, 
$$d(a)\geq\bigvee{p\in W} e(p)$$
Thus, $d\geq m$.  Thus,
$$m = \bigvee W$$
Let $p$ and $q$ be two frame arrows from $S$ to $L$.  Let $b$ map $a$ to $p(a)\wedge q(a)$.  Then, $b(a)\leq p(a)$ and $b(a)\leq q(a)$ so $b\leq p$ and $b\leq q$.  Suppose $h\leq p$ and $h\leq q$. Then, $h(a)\leq p(a)$ and $h(a)\leq q(a)$, Thus, $h(a)\leq p(a)\wedge q(a) = b(a)$. Thus, $h\leq b$.  Thus, 
$$b = p\wedge q$$\\
Now, $e:p\mapsto p(a)$ defines a function $\mbox{Arr}\langle S,L\rangle\to L$ and, as we have just seen, if $W\subseteq\mbox{Arr}\langle S,L\rangle$, 
$$e(\bigvee W) = \bigvee_{p\in W}e(p)$$
and if $p$ and $q$ are elements of $\mbox{Arr}\langle S,L\rangle$,
$$e(p\wedge q) = e(p)\wedge e(q)$$.
Thus, $e$ is an arrow of frames.\\
Now define $z:L\to \mbox{Arr}\langle S,L\rangle$ by 
$$l\mapsto \left(p: a\mapsto l\right)$$
By \ref{agoesanywhere}, $p$ is an element of $\mbox{Arr}\langle S,L\rangle$.  Suppose $l\leq x$ in $L$. Then, $z(l) = d:a\mapsto l\leq z(x) = c:a\mapsto x$.  The proof is completed by Lemma \ref{partordisomimplisframeisom}.  
\end{proof}
In the next sections we prove that 
\begin{enumerate}
\item $S$ is, up to isomorphism, the only object in the category of frames for which theorem \ref{Sis2pt3autogen} is true.  We will do this by showing:
\begin{enumerate}
\item That in the case when the locale has precisely four elements, there are greater than 3 automorphisms
\item Any locale with five or more elements which is a generator and 2 pointed has more than 3 automorphisms.
\end{enumerate}
\item Any equivalence on the category of frames preserves
\begin{enumerate}
\item The number of points of a locale
\item The number of automorphisms on a locale
\item The property of being a generator
\item The order on the sets of arrows.
\end{enumerate}
\end{enumerate}
\section{Equivalences on a Category}
In this section we recall some basic facts about functors, focusing especially on a particular kind of functorial equivalences.  Throughout the paper, I will distinguish functors notationally by rendering them in black board font, for example, ``the functor $\mathbbm{F}$''. 
\begin{definition}
Let $\mathcal{C}$ and $\mathcal{D}$ be categories.  A functor from $\mathcal{C}$ to $\mathcal{D}$ is an assignment carrying objects of $\mathcal{C}$ to objects of $\mathcal{D}$ and a function
$$\mbox{Arr}(C,C')\to \mbox{Arr}(\mathbbm{F}C,\mathbbm{F}C')$$
for every pair of objects $C$ and $C'$ of $\mathcal{C}$, subject to the condition that composition and the identity arrows are preserved.
\end{definition}
\begin{definition}
A functor $\mathbbm{F}:\mathcal{C}\to\mathcal{D}$ is faithful if the function 
$$\mbox{Arr}(C,C')\to \mbox{Arr}(FC,FC')$$
is one-to-one for every pair of objects $C$ and $C'$ of $\mathcal{C}$.
\end{definition}
\begin{definition}
A functor $\mathbbm{F}:\mathcal{C}\to\mathcal{D}$ is full if the function 
$$\mbox{Arr}(C,C')\to \mbox{Arr}(FC,FC')$$
is onto for every pair of objects $C$ and $C'$ of $\mathcal{C}$.
\end{definition}
\begin{definition}
Let $\mathbbm{T}$ and $\mathbbm{R}$ be functors from $\mathcal{C}$ to $\mathcal{D}$.  A natural transformation from $\mathbbm{T}$ to $\mathbbm{R}$, denoted $\alpha: T\Rightarrow R$,  is a class of arrows $\alpha_{C}:TC\to RC$ for every object $C$ of $\mathcal{C}$ so that the diagram\\
\centerline{
\xymatrix{
\mathbbm{T}C\ar[r]^{\alpha_{C}}\ar[d]_{Tf} & \mathbbm{R}C\ar[d]^{Rf}\\
\mathbbm{T}C'\ar[r]_{\alpha_{C'}} & \mathbbm{R}C'
}
}
commutes for every pair of objects $C$, $C'$ and every arrow $f:C\to C'$.
\end{definition}
\begin{definition}
A natural transformation $\alpha:\mathbbm{T}\Rightarrow \mathbbm{R}$ is a natural isomorphism if $\alpha_{C}$ is an isomorphism for every object $C$ of $\mathcal{C}$.
\end{definition}
\begin{definition}
Let $\mathcal{C}$ and $\mathcal{D}$ be categories.  A functor $\mathbbm{T}: \mathcal{C}\to \mathcal{D}$ is a left adjoint of $\mathbbm{R}:\mathcal{D}\to \mathcal{C}$ if for every pair of objects $C$ of $\mathcal{C}$ and $D$ of $\mathcal{D}$ there exists a bijection 
$$\theta_{C,D}:\mbox{Arr}(\mathbbm{T}C,D)\to\mbox{Arr}(C,\mathbbm{R}D)$$
which is natural in both $C$ and $D$.
\end{definition}

\begin{definition}\label{defeq}
Let $\mathcal{C}$ be a category.  An equivalence on $\mathcal{C}$ is a functor $\mathbbm{T}:\mathcal{C}\to \mathcal{C}$ which is full and faithful and has a full and faithful left adjoint $\mathbbm{R}$.
\end{definition}
\begin{lemma}\label{equivequiv}
A functor $\mathbbm{T}:\mathcal{C}\to \mathcal{D}$ is an equivalence  if and only if any of the following hold:
\begin{enumerate}
\item $\mathbbm{T}$ has a full and faithful right adjoint $\mathbbm{R}$.
\item $\mathbbm{T}$ has a right adjoint $\mathbbm{R}$ and the two canonical natural transformations of the adjunction $\eta:1_{\mathcal{C}}\Rightarrow \mathbbm{R}\mathbbm{T}$ and $\epsilon : 1_{\mathcal{D}}\Rightarrow \mathbbm{T}\mathbbm{R}$ are isomorphisms.
\item There exists a functor $\mathbbm{R}:\mathcal{D}\to \mathcal{C}$ and two natural isomorphisms $\alpha :1_{\mathcal{C}}\Rightarrow \mathbbm{R}\mathbbm{T}$ and $\beta : 1_{\mathcal{D}}\Rightarrow \mathbbm{T}\mathbbm{R}$ 
\item $\mathbbm{T}$ is full and faithful and each object $D$ of $\mathcal{D}$ is isomorphic to an object of the form $\mathbbm{T}A$ for some object $A$ of $\mathcal{C}$.
\end{enumerate}
\end{lemma}
\begin{lemma}\label{isomimpbij}
Let $X$ and $Y$ be objects in a category $\mathcal{C}$. Suppose $X$ is isomorphic to $Y$. Let $C$ be any object in $\mathcal{C}$. Then the cardinality of the set $\mbox{Arr}\langle C,X\rangle$ is the same as the cardinality of the set $\mbox{Arr}\langle C,Y\rangle$ and the cardinality of the set $\mbox{Arr}\langle C,X\rangle$ is the same as the cardinality of the set $\mbox{Arr}\langle C,Y\rangle$.
\end{lemma}
\begin{proof}

Let $X$ and $Y$ be objects in category $\mathcal{C}$. Suppose $X$ is isomorphic to $Y$. Let $C$ be an arbitrary object in $\mathcal{C}$. Let $\phi:X\to Y$ be the isomorphism.  Define $\Phi:  \mbox{Arr}\langle C,Y\rangle\to \mbox{Arr}\langle C,X\rangle$  by $f\mapsto \phi\circ f$.  Let $f\in\mbox{Arr}\langle C,X\rangle$ and $g\in\mbox{Arr}\langle C,X\rangle$ be arbitrary.  Suppose $\Phi(f) = \Phi(g)$.  Then, $\phi\circ f = \phi\circ g$. Since $\phi$ is an isomorphism, it is monic.  Thus, $f = g$.  Thus, $\Phi$ is one-to-one.  Let $h$ be arbitrary and suppose $h\in\mbox{Arr}\langle C,Y\rangle$. Since $\phi $ is an isomorphism, there exists an arrow $\psi$ so that $\phi\circ\psi = 1_{Y}$ and $\psi\circ\phi = 1_{X}$.  Then, $\Phi(\psi h) = \phi\circ\psi\circ h = 1_{X}\circ h = h$.  Thus, $\Phi$ is onto.  Since $\Phi$ is one-to-one and onto it is a bijection of sets.  Thus the cardinality of the set $\mbox{Arr}\langle C,X\rangle$ is the same as the cardinality of the set $\mbox{Arr}\langle C,Y\rangle$.  The second assertion follows by defining $\Phi$ by $f\mapsto f\circ\phi$ and arguing as before. 
\end{proof}
\begin{remark}
Let $\mathbbm{T}$ be an equivalence from a category $\mathcal{C}$ to $\mathcal{D}$.  Then, since $\mathbbm{T}$ is full and faithful(\ref{equivequiv}) the assignment $f\mapsto \mathbbm{T}f$ defines a bijection of sets $\mbox{Arr}\langle A,B\rangle\to\mbox{Arr}\langle \mathbbm{T}A,\mathbbm{T}B\rangle$.
\end{remark}
\begin{proposition}\label{eqpresterm}
Let $\mathbbm{T}$ be an equivalence on a category $\mathcal{C}$.  Suppose $\Omega$ is the terminal object of $\mathcal{C}$.  Then $\mathbbm{T}\Omega$ is also a terminal object of $\mathcal{C}$.
\end{proposition}
\begin{proof}
Let the adjunct of $\mathbbm{T}$ be $\mathbbm{S}$.  Then, we have bijections of sets:
$$\mbox{Arr}\langle X,\mathbbm{T}\omega\rangle\to\mbox{Arr}\langle \mathbbm{T}\mathbbm{R}X,\mathbbm{T}\omega\rangle\to\mbox{Arr}\langle \mathbbm{R}X,\mathbbm{R}\mathbbm{T}\omega\rangle\to\mbox{Arr}\langle X,\mathbbm{T}\omega\rangle$$
the last of which has just one element since $\Omega $ is terminal.  Thus, so too does the first set.  Thus, since $X$ was arbitrary, $\mathbbm{T}\Omega$ is a terminal object.
\end{proof}
\begin{theorem}\label{eqprespts}
Suppose $\mathbbm{T}$ is an equivalence on the category of frames.  Let $L$ be a frame.  Then the cardinality of the set of points of $L$ is the same as the cardinality of the set of points of $\mathbbm{T}L$.
\end{theorem}
\begin{proof}
Let $c$ be the cardinality of the set of points of $L$.  Then, by definition, 
\begin{align}
c &= |\mbox{Arr}\langle L,T\rangle|\\\label{cardpts1}
  &= |\mbox{Arr}\langle \mathbbm{T}L,\mathbbm{T}T\rangle|\\\label{cardpts2}
  &= |\mbox{Arr}\langle \mathbbm{T}L, T\rangle|\\\label{cardpts3}
\end{align}
where \ref{cardpts1} follows from definition\ref{defpts}, \ref{cardpts2} follows from \ref{eqpresterm} comgined with \ref{isomimpbij}, by definition, \ref{cardpts3} is the cardinality of the set of points of $\mathbbm{T}L$.
\end{proof}
\begin{proposition}
Let $G$ be a generator in a category $\mathcal{C}$.  Let $\mathbbm{T}$ be an equivalence on $\mathcal{C}$.  Then, $\mathbbm{T}G$ is also a generator.
\end{proposition}
\begin{proof}
Let\\
\centerline{
\xymatrix{
A\ar@<.5ex>[r]^{f}\ar@<-.5ex>[r]_{g} & B
}
}
be arbitrary arrows in $\mathcal{C}$ and suppose $f\ne g$.  Then since $\mbox{Arr}\langle A,B\rangle$ is bijective to $\mbox{Arr}\langle \mathbbm{R}A,\mathbbm{R}B\rangle$, $\mathbbm{R}f\ne \mathbbm{R}g$.  Thus,since $G$ is a generator, there exists an arrow $\phi:G\to RA$ so that $\mathbbm{R}f\circ\phi\ne Rg\circ\phi$.  Since $\mathbbm{T}$ is an equivalence it possesses a left adjoint $\mathbbm{R}$. Thus $\mbox{Arr}\langle \mathbbm{T}G,B\rangle$ is bijective to $\mbox{Arr}\langle G,\mathbbm{R}B\rangle$.  Thus, there exists an arrow $\psi$ which is the image of $\phi$ under this bijection so that $f\circ\psi\ne g\circ\psi$.  Thus, since $A$, $B$, $f$ and $g$ were all arbitrary, $\mathbbm{T}G$ is a generator in $\mathcal{C}$.  
\end{proof}
Here we prove a brief lemma which is primarily a side-note
\begin{lemma}
Let $\mathcal{L}$ be a category.  Let $\mathbbm{T}$ be an equivalence on $\mathcal{L}$ and $\mathbbm{R}$ its right adjoint.  Then, $\mathbbm{R}\mathbbm{T}$ and $\mathbbm{T}\mathbbm{R}$ are naturally isomorphic.
\end{lemma}
\begin{proof}
Let $f:A\to B$ be an arrow in $\mathcal{L}$.  Then,\\
\centerline{
\xymatrix{
\mathbbm{T}\mathbbm{R}A\ar[d]_{\mathbbm{T}\mathbbm{R}f} & A\ar[l]_{\epsilon_{A}}\ar[d]|{f}\ar[l]_{\epsilon_{A}}\ar[r]^{\eta{A}} & \mathbbm{R}\mathbbm{T}\ar[d]^{\mathbbm{R}\mathbbm{T}f}\\
\mathbbm{T}\mathbbm{R}B & B\ar[l]^{\epsilon_{B}}\ar[r]_{\eta_{B}} & \mathbbm{R}\mathbbm{T}B
}
}
commutes, so we have that
$$\eta\epsilon^{-1}:\mathbbm{T}\mathbbm{R}\Rightarrow\mathbbm{R}\mathbbm{T}$$
is an isomorphism.
\end{proof}
\begin{theorem}\label{eqpressums}
Let $\mathbbm{T}$ be an equivalence on a category $\mathcal{L}$.  If \\
\centerline{
\xymatrix{
A\ar[dr]^{i}\\
& A+A\\
A\ar[ur]_{j}
}
}
is a sum diagram, so is\\
\centerline{
\xymatrix{
\mathbbm{T}A\ar[dr]^{\mathbbm{T}i}\\
& \mathbbm{T}(A+A)\\
\mathbbm{T}A\ar[ur]_{\mathbbm{T}j}
}
}
\end{theorem}
\begin{proof}
Let $Z$ be any object of $\mathcal{L}$ and consider the diagram\\
\centerline{
\xymatrix{
\mathbbm{T}A\ar[dr]^{\mathbbm{T}i}\ar@/^/[drr]^{w}\\
& \mathbbm{T}(A+A)& Z\\
\mathbbm{T}A\ar[ur]_{\mathbbm{T}j}\ar@/_/[urr]_{z}
}
}
Then, we have\\
\centerline{
\xymatrix{
A\ar[r]^{\eta_{A}}&\mathbbm{R}\mathbbm{T}A\ar[dr]^{\mathbbm{T}i}\ar@/^/[drr]^{w}\\
&& \mathbbm{T}(A+A)& Z\\
A\ar[r]^{\eta_{A}}&\mathbbm{R}\mathbbm{T}A\ar[ur]_{\mathbbm{T}j}\ar@/_/[urr]_{z}
}
}
Thus,\\
\centerline{
\xymatrix{
A\ar[dr]|{i}\ar@/^/[drr]^{Rw\circ\eta_{A}}\\
& A+A & \mathbbm{R}Z\\
A\ar[ur]|{j}\ar@/_/[urr]^{Rz\circ\eta_{A}}\\
}
}
and so there exists a unique $\phi$ so that\\
\centerline{
\xymatrix{
A\ar[dr]|{i}\ar@/^/[drr]^{Rw\circ\eta_{A}}\\
& A+A\ar@*{[red]}[r]|{\phi} & \mathbbm{R}Z\\
A\ar[ur]|{j}\ar@/_/[urr]_{Rz\circ\eta_{A}}\\
}
}
commutes.  Then, we have the diagram\\
\centerline{
\xymatrix{
\mathbbm{T}A\ar[dr]|{\mathbbm{T}i}\ar@/^/[drrr]^{w}\\
& \mathbbm{T}(A+A)\ar@*{[red]}[r]|{\mathbbm{T}\phi} & \mathbbm{T}\mathbbm{R}Z\ar[r]|{\epsilon^{1}_{Z}} & Z\\
\mathbbm{T}A\ar[ur]|{\mathbbm{T}j}\ar@/_/[urrr]_{z}\\
}
}
which commutes since\\
\begin{align}
\epsilon^{-1}_{Z}\mathbbm{T}\phi\mathbbm{T}i &= \epsilon^{-1}_{Z}\mathbbm{T}\mathbbm{R}w\mathbbm{T}\eta_{A}\\
&= w\epsilon^{-1}_{\mathbbm{T}A}\\
&= w
\end{align}
by commutativity of\\
\centerline{
\xymatrix{
\mathbbm{T}A\ar[r]^{\epsilon_{\mathbbm{T}A}}\ar[d]_{w} & \mathbbm{T}\mathbbm{R}A\ar[d]|{\mathbbm{T}\mathbbm{R}w} & \mathbbm{T}A\ar[d]^{w}\ar[l]_{\mathbbm{T}\eta_{A}}\\
Z\ar[r]_{\epsilon_{Z}} & \mathbbm{T}\mathbbm{R}Z & Z\ar[l]_{\epsilon^{-1}_{Z}}
}
}
A similar argument shows that 
$$\epsilon^{-1}_{Z}\mathbbm{T}\phi\mathbbm{T}j = z$$
Now, suppose
\begin{align}
\psi\mathbbm{T}i &= w
\end{align}
Then,
\begin{align}
\mathbbm{R}\psi\mathbb{R}\mathbbm{T}i &= \mathbbm{R}w\\
\mathbbm{R}\psi\mathbb{R}\mathbbm{T}i\eta_{A} &= \mathbbm{R}w\eta_{A}\\
\mathbbm{R}\psi\eta_{A+A}i &= \mathbbm{R}w\eta_{A}\\
\mathbbm{R}\psi\eta_{A+A} &= \phi\\
\mathbbm{T}\mathbbm{R}\psi\eta_{A+A} = \mathbbm{T}\phi\\
\mathbbm{T}\mathbbm{R}\psi\eta_{A+A}\epsilon^{-1}_{Z} = \mathbbm{T}\phi\epsilon^{-1}_{Z}\\
\psi = \mathbbm{T}\phi\epsilon^{-1}_{Z}
\end{align}
by commutativity of the diagrams\\
\centerline{
\xymatrix{
A\ar[r]^{\eta_{A}}\ar[d]_{i} & \mathbbm{R}\mathbbm{T}A\ar[d]^{\mathbbm{R}\mathbbm{T}i}\\
A+A\ar[r]_{\eta_{A+A}} & \mathbbm{R}\mathbbm{T}(A+A)
}
}
and
\centerline{
\xymatrix{
\mathbbm{T}(A+A)\ar[r]^{\eta_{A+A}}\ar[d]_{i} & \mathbbm{T}\mathbbm{R}\mathbbm{T}(A+A)\ar[d]\\
Z\ar[r]_{\epsilon_{Z}} & \mathbbm{T}\mathbbm{T}(Z)
}
}
\end{proof}
\begin{corollary}
Let $\triangledown_{A}$ be the codiagonal arrow of the object $A$ in the category $\mathcal{L}$.  Let $\mathbbm{T}$ be an equivalence on $\mathcal{L}$.  Then,
$$\mathbbm{T}\triangledown_{A} = \triangledown_{\mathbbm{T}A}$$
\end{corollary}
\begin{proof}
Recall that, by definition, $\triangledown_{A}$ is the unique arrow which makes the diagram\\
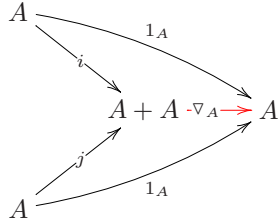
\centerline{
\xymatrix{
A\ar[dr]|{i}\ar@/^/[drr]^{1_{A}}\\
& A+A\ar@*{[red]}[r]|{\triangledown_{A}} & A\\
A\ar[ur]|{j}\ar@/_/[urr]_{1_{A}}
}
}
commute, where $A+A$ is the sum of $A$ with itself. Functors preserve composition and identity arrows. So, \\
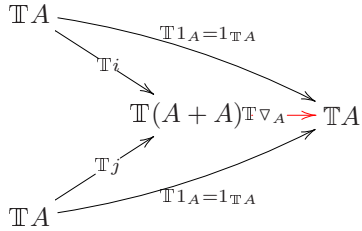
\centerline{
\xymatrix{
\mathbbm{T}A\ar[dr]|{\mathbbm{T}i}\ar@/^/[drr]^{\mathbbm{T}1_{A}=1_{\mathbbm{T}A}}\\
&\mathbbm{T}(A+A)\ar@*{[red]}[r]|{\mathbbm{T}\triangledown_{A}} & \mathbb{T}A\\
\mathbbm{T}A\ar[ur]|{\mathbbm{T}j}\ar@/_/[urr]_{\mathbbm{T}1_{A}=1_{\mathbbm{T}A}}
}
}
commutes. But, $\triangledown_{\mathbbm{T}A}$ is the unique arrow so that\\
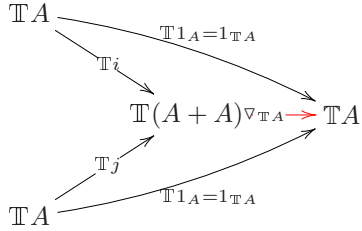
\centerline{
\xymatrix{
\mathbbm{T}A\ar[dr]|{\mathbbm{T}i}\ar@/^/[drr]^{\mathbbm{T}1_{A}=1_{\mathbbm{T}A}}\\
&\mathbbm{T}(A+A)\ar@*{[red]}[r]|{\triangledown_{\mathbbm{T}A}} & \mathbb{T}A\\
\mathbbm{T}A\ar[ur]|{\mathbbm{T}j}\ar@/_/[urr]_{\mathbbm{T}1_{A}=1_{\mathbbm{T}A}}
}
}
commute.  Thus, 
$$\mathbbm{T}\triangledown_{A} = \triangledown_{\mathbbm{T}A}$$
\end{proof}
We end this section with the last of the technical lemmas having to do with equivalences.
\begin{lemma}
Let\\
\centerline{
\xymatrix{
A\ar@<.5ex>[r]^{f}\ar@<-.5ex>[r]_{g} & B
}
}
be a pair of arrows in a category and let $\mathbbm{T}$ be an equivalence on that category.  Let $f+g$ be the unique arrow so that\\
\centerline{
\xymatrix{
A\ar[dr]^{i_{S}}\ar[rrr]^{f} &&& B\ar[dl]_{i_{L}}\\
& A+A\ar[r]^{f+g} & B+B\\
A\ar[ur]^{j_{S}}\ar[rrr]_{g} &&& B\ar[ul]^{j_{L}}
}
}
commutes.  Then, $\mathbbm{T}(f+g) = \mathbbm{T}(f)+\mathbbm{T}(g)$.
\end{lemma}
\begin{proof}
Because $\mathbbm{T}$ preserves composition and because (\ref{eqpressums}) $\mathbbm{T}(A+A) = \mathbbm{T}(A)+\mathbbm{T}(A)$, we have that\\
\centerline{
\xymatrix{
\mathbbm{T}A\ar[dr]^{\mathbbm{T}i_{\mathbbm{T}A}}\ar[rrr]^{\mathbbm{T}f} &&& \mathbbm{T}B\ar[dl]_{\mathbbm{T}i_{\mathbbm{T}B}}\\
& \mathbbm{T}A+A\ar[r]^{\mathbbm{T}(f+g)} & \mathbbm{T}(B+B)\\
\mathbbm{T}A\ar[ur]^{\mathbbm{T}j_{\mathbbm{T}A}}\ar[rrr]_{\mathbbm{T}g} &&& \mathbbm{T}B\ar[ul]^{\mathbbm{T}j_{\mathbbm{T}B}}
}
} 
is a commutative diagram of the sums of $\mathbbm{T}A$ with itself and $\mathbbm{T}B$ with itself.  But $\mathbbm{T}f + \mathbbm{T}g$ is the unique such arrow from $\mathbbm{T}A+\mathbbm{T}A$ to $\mathbbm{T}B+\mathbbm{T}B$ making the diagram commute.  Thus it must be that 
$$\mathbbm{T}(f+g) = \mathbbm{T}(f)+\mathbbm{T}(g)$$
\end{proof}

\section{If $L$ has Four Elements}
In this section we show that if $L$ has four elements then $L$ has four or more automorphisms.
\begin{proposition}
Let $L$ be a locale and suppose that the cardinality of $L$, (which we will denote by $|L|$) is equal to four.  Then $\mbox{Arr}\langle L,L\rangle\geq 4$ as well.
\end{proposition}
\begin{proof}
We proceed by cases:\\
\underline{Case I:}Suppose $L = \lbrace 0 < a < b < 1\rbrace$ is a total order. The following functions are all arrows in the category of frames.
\begin{enumerate}
\item 
\begin{equation}
1_{L}(x) = \left\{ \begin{array}{cc}
               1 & \mbox{for $x = 1$}\\
               a & \mbox{for $x = a$}\\
               b & \mbox{for $x = b$}\\
               0 & \mbox{for $x = 0$}
               \end{array}
\right.
\end{equation}
\item 
\begin{equation}
f(x) = \left\{ \begin{array}{cc}
               1 & \mbox{for $x = 1$}\\
               1 & \mbox{for $x = a$}\\
               0 & \mbox{for $x = b$}\\
               0 & \mbox{for $x = 0$}
               \end{array}
\right.
\end{equation}
We check to see that $f$ preserves joins and meets:
\begin{enumerate}
\item
$$f(1\wedge a) = f(a) = 1 = 1\wedge 1 = f(1)\wedge f(a)$$
\item
$$f(1\wedge b) = f(b) = 0 = 1\wedge 0 = f(1)\wedge f(b)$$
\item
$$f(1\wedge 0) = f(0) = 0 = 1\wedge 0 = f(1)\wedge f(0)$$
\item
$$f(a\wedge b) = f(b) = 0 = 1\wedge 0 = f(a)\wedge f(b)$$
\item
$$f(a\wedge 0) = f(0) = 0 = 1\wedge 0 = f(a)\wedge f(0)$$
\item
$$f(b\wedge 0 ) = f(0) = 0 = 0\wedge 0 = f(b)\wedge f(0)$$
\end{enumerate}
\begin{enumerate}
\item
$$f(1\vee a) = f(1) = 1 = 1\vee 1 = f(1)\vee f(a)$$
\item
$$f(1\vee b) = f(1) = 1 = 1\vee 0 = f(1)\vee f(b)$$
\item
$$f(1\vee 0) = f(1) = 1 = 1\vee 0 = f(1)\vee f(0)$$
\item
$$f(a\vee b) = f(a) = 1 = 1\vee 0 = f(a)\vee f(b)$$
\item
$$f(a\vee 0) = f(a) = 1 = 1\vee 0 = f(a)\vee f(0)$$
\item
$$f(b\vee 0 ) = f(b) = 0 = 0\vee 0 = f(b)\vee f(0)$$
\end{enumerate}
\item 
\begin{equation}
g(x) = \left\{ \begin{array}{cc}
               1 & \mbox{for $x = 1$}\\
               1 & \mbox{for $x = a$}\\
               1 & \mbox{for $x = b$}\\
               0 & \mbox{for $x = 0$}
               \end{array}
\right.
\end{equation}
We check to see that $g$ preserves joins and meets:
\begin{enumerate}
\item
$$g(1\wedge a) = g(a) = 1 = 1\wedge 1 = g(1)\wedge g(a)$$
\item
$$g(1\wedge b) = g(b) = 1 = 1\wedge 1 = g(1)\wedge g(b)$$
\item
$$g(1\wedge 0) = g(0) = 0 = 1\wedge 0 = g(1)\wedge g(0)$$
\item
$$g(a\wedge b) = g(b) = 1 = 1\wedge 1 = g(a)\wedge g(b)$$
\item
$$g(a\wedge 0) = g(0) = 0 = 1\wedge 0 = g(a)\wedge g(0)$$
\item
$$g(b\wedge 0 ) = g(0) = 0 = 0\wedge 0 = g(b)\wedge g(0)$$
\end{enumerate}
\begin{enumerate}
\item
$$g(1\vee a) = g(1) = 1 = 1\vee 1 = g(1)\vee g(a)$$
\item
$$g(1\vee b) = g(1) = 1 = 1\vee 1 = g(1)\vee g(b)$$
\item
$$g(1\vee 0) = g(1) = 1 = 1\vee 0 = g(1)\vee g(0)$$
\item
$$g(a\vee b) = g(a) = 1 = 1\vee 1 = g(a)\vee g(b)$$
\item
$$g(a\vee 0) = g(a) = 1 = 1\vee 0 = g(a)\vee g(0)$$
\item
$$g(b\vee 0 ) = g(b) = 1 = 1\vee 0 = g(b)\vee g(0)$$
\end{enumerate}
\item 
\begin{equation}
h(x) = \left\{ \begin{array}{cc}
               1 & \mbox{for $x = 1$}\\
               0 & \mbox{for $x = a$}\\
               0 & \mbox{for $x = b$}\\
               0 & \mbox{for $x = 0$}
               \end{array}
\right.
\end{equation}
We check to see that $h$ preserves joins and meets:
\begin{enumerate}
\item
$$h(1\wedge a) = h(a) = 0 = 1\wedge 0 = h(1)\wedge h(a)$$
\item
$$h(1\wedge b) = h(b) = 0 = 1\wedge 0 = h(1)\wedge h(b)$$
\item
$$h(1\wedge 0) = h(0) = 0 = 1\wedge 0 = h(1)\wedge h(0)$$
\item
$$h(a\wedge b) = h(b) = 0 = 1\wedge 0 = h(a)\wedge h(b)$$
\item
$$h(a\wedge 0) = h(0) = 0 = 1\wedge 0 = h(a)\wedge h(0)$$
\item
$$h(b\wedge 0 ) = h(0) = 0 = 0\wedge 0 = h(b)\wedge h(0)$$
\end{enumerate}
and joins:
\begin{enumerate}
\item
$$h(1\vee a) = h(1) = 1 = 1\vee 0 = h(1)\vee h(a)$$
\item
$$h(1\vee b) = h(1) = 1 = 1\vee 0 = h(1)\vee h(b)$$
\item
$$h(1\vee 0) = h(1) = 1 = 1\vee 0 = h(1)\vee h(0)$$
\item
$$h(a\vee b) = h(a) = 0 = 0\vee 0 = h(a)\vee h(b)$$
\item
$$h(a\vee 0) = h(a) = 0 = 0\vee 0 = h(a)\vee h(0)$$
\item
$$h(b\vee 0 ) = h(b) = 0 = 0\vee 0 = h(b)\vee h(0)$$
\end{enumerate}
\end{enumerate}  
\underline{Case II} Suppose $L$ is not a total order.  Then, L has the form\\
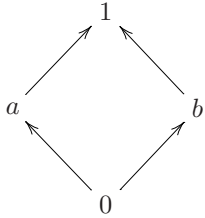
\centerline{
\xymatrix{
& 1\\
a\ar[ur]&& b\ar[ul]\\
& 0 \ar[ur]\ar[ul]
}
}
Let
\begin{enumerate}
\item 
\begin{equation}
1_{L}(x) = \left\{ \begin{array}{cc}
               1 & \mbox{for $x = 1$}\\
               a & \mbox{for $x = a$}\\
               b & \mbox{for $x = b$}\\
               0 & \mbox{for $x = 0$}
               \end{array}
\right.
\end{equation}
\item 
\begin{equation}
f(x) = \left\{ \begin{array}{cc}
               1 & \mbox{for $x = 1$}\\
               1 & \mbox{for $x = a$}\\
               0 & \mbox{for $x = b$}\\
               0 & \mbox{for $x = 0$}
               \end{array}
\right.
\end{equation}
We check to see that $f$ preserves joins and meets:
\begin{enumerate}
\item
$$f(1\wedge a) = f(a) = 1 = 1\wedge 1 = f(1)\wedge f(a)$$
\item
$$f(1\wedge b) = f(b) = 0 = 1\wedge 0 = f(1)\wedge f(b)$$
\item
$$f(1\wedge 0) = f(0) = 0 = 1\wedge 0 = f(1)\wedge f(0)$$
\item
$$f(a\wedge b) = f(0) = 0 = 1\wedge 0 = f(a)\wedge f(b)$$
\item
$$f(a\wedge 0) = f(0) = 0 = 1\wedge 0 = f(a)\wedge f(0)$$
\item
$$f(b\wedge 0 ) = f(0) = 0 = 0\wedge 0 = f(b)\wedge f(0)$$
\end{enumerate}
\begin{enumerate}
\item
$$f(1\vee a) = f(1) = 1 = 1\vee 1 = f(1)\vee f(a)$$
\item
$$f(1\vee b) = f(1) = 1 = 1\vee 0 = f(1)\vee f(b)$$
\item
$$f(1\vee 0) = f(1) = 1 = 1\vee 0 = f(1)\vee f(0)$$
\item
$$f(a\vee b) = f(1) = 1 = 1\vee 0 = f(a)\vee f(b)$$
\item
$$f(a\vee 0) = f(a) = 1 = 1\vee 0 = f(a)\vee f(0)$$
\item
$$f(b\vee 0 ) = f(b) = 0 = 0\vee 0 = f(b)\vee f(0)$$
\end{enumerate}
\item 
\begin{equation}
g(x) = \left\{ \begin{array}{cc}
               1 & \mbox{for $x = 1$}\\
               0 & \mbox{for $x = a$}\\
               1 & \mbox{for $x = b$}\\
               0 & \mbox{for $x = 0$}
               \end{array}
\right.
\end{equation}
We check to see that $g$ preserves joins and meets:
\begin{enumerate}
\item
$$g(1\wedge a) = g(a) = 0 = 1\wedge 0 = g(1)\wedge g(a)$$
\item
$$g(1\wedge b) = g(b) = 1 = 1\wedge 1 = g(1)\wedge g(b)$$
\item
$$g(1\wedge 0) = g(0) = 0 = 1\wedge 0 = g(1)\wedge g(0)$$
\item
$$g(a\wedge b) = g(b) = 0 = 0\wedge 1 = g(a)\wedge g(b)$$
\item
$$g(a\wedge 0) = g(0) = 0 = 1\wedge 0 = g(a)\wedge g(0)$$
\item
$$f(b\wedge 0 ) = g(0) = 0 = 0\wedge 0 = f(b)\wedge f(0)$$
\end{enumerate}
\begin{enumerate}
\item
$$g(1\vee a) = g(1) = 1 = 1\vee 1 = g(1)\vee g(a)$$
\item
$$g(1\vee b) = g(1) = 1 = 1\vee 1 = g(1)\vee g(b)$$
\item
$$g(1\vee 0) = g(1) = 1 = 1\vee 0 = g(1)\vee g(0)$$
\item
$$g(a\vee b) = g(1) = 1 = 0\vee 1 = g(a)\vee g(b)$$
\item
$$g(a\vee 0) = g(a) = 0 = 0\vee 0 = g(a)\vee g(0)$$
\item
$$g(b\vee 0 ) = g(b) = 1 = 1\vee 0 = g(b)\vee g(0)$$
\end{enumerate}
\item 
\begin{equation}
h(x) = \left\{ \begin{array}{cc}
               1 & \mbox{for $x = 1$}\\
               b & \mbox{for $x = a$}\\
               a & \mbox{for $x = b$}\\
               0 & \mbox{for $x = 0$}
               \end{array}
\right.
\end{equation}
We check to see that $h$ preserves joins and meets:
\begin{enumerate}
\item
$$h(1\wedge a) = h(a) = b = 1\wedge b = h(1)\wedge h(a)$$
\item
$$h(1\wedge b) = h(b) = a = 1\wedge a = h(1)\wedge h(b)$$
\item
$$h(1\wedge 0) = h(0) = 0 = 1\wedge 0 = h(1)\wedge h(0)$$
\item
$$h(a\wedge b) = h(0) = 0 = b\wedge a = h(a)\wedge h(b)$$
\item
$$h(a\wedge 0) = h(0) = 0 = b\wedge 0 = h(a)\wedge h(0)$$
\item
$$h(b\wedge 0 ) = h(0) = 0 = a\wedge 0 = h(b)\wedge h(0)$$
\end{enumerate}
\begin{enumerate}
\item
$$h(1\vee a) = h(1) = 1 = 1\vee b = h(1)\vee h(a)$$
\item
$$h(1\vee b) = h(1) = 1 = 1\vee a = h(1)\vee h(b)$$
\item
$$h(1\vee 0) = h(1) = 1 = 1\vee 0 = h(1)\vee h(0)$$
\item
$$h(a\vee b) = h(1) = 1 = b\vee a = h(a)\vee h(b)$$
\item
$$h(a\vee 0) = h(a) = b = b\vee 0 = h(a)\vee h(0)$$
\item
$$h(b\vee 0 ) = h(b) = a = a\vee 0 = h(b)\vee h(0)$$
\end{enumerate}
\end{enumerate}  
Since a locale is either a total order or it's not, these are the only two possible locales with four elements, and, since, in each case there are at least four arrows from $L$ to $L$, if $L$ has four elements then there are at least four arrows in the category of frames from $L$ to $L$.
\end{proof}
\section{If $L$ Has Five or More Elements}
As we have just seen, if $L$ has four elements, $L$ also has at least four automorphisms. Proving that the same is true for locales with more than four elements is, seemingly, much more difficult(if it is, indeed, true at all).  We instead prove something a bit easier.  Namely
\begin{theorem}
Suppose $L$ has five or more elements. Suppose $L$ is a two-pointed generator in the category of frames.  Then, $L$ has at least four automorphisms. 
\end{theorem}
\begin{proof}
Suppose $L$ is a locale and that $L$ has 2 points and is a generator in the category of frames. Let $x$ and $y$ be the prime elements of $L$ corresponding to the points $p_{1}$ and $p_{2}$.  Then, we have the rules:
\begin{equation}
p_{1}(l) = \left\{ \begin{array}{cc}
               1 & \mbox{for $l\nleq x$}\\
               0 & \mbox{for $x\leq x$}
               \end{array}
\right.
\end{equation}
\begin{equation}
p_{2}(l) = \left\{ \begin{array}{cc}
               1 & \mbox{for $l\nleq y$}\\
               0 & \mbox{for $l\leq y$}
               \end{array}
\right.
\end{equation}
Let $f$ and $g$ be the two points of the Sierpinski locale $S$.  That is,
\begin{equation}
f(x) = \left\{ \begin{array}{cc}
               1 & \mbox{for $x=1$}\\
               0 & \mbox{for $x=a$}\\
               0 & \mbox{for $x=0$}   
               \end{array}
\right.
\end{equation}
\begin{equation}
g(x) = \left\{ \begin{array}{cc}
               1 & \mbox{for $x=1$}\\
               1 & \mbox{for $x=a$}\\
               0 & \mbox{for $x=0$}\\
               \end{array}
\right.
\end{equation}
Now, $f\ne g$.  Since $L$ is a generator we can find $\phi:L\to S$ so that $f\circ\phi\ne g\circ\phi$.  This implies that we can find $z\in L$ so that $\phi(z) = a$. Note that this means that $z\ne 0$ and $z\ne 1$.  Note, too, that $f\phi$ and $g\phi$ are both points of $L$. Since $f\phi(z) = 0$ and $g\phi(z) = 1$, these points are distinct.  Thus, either $f\phi = p_{1}$ and $g\phi = p_{2}$ or $f\phi = p_{2}$ and $g\phi = p_{1}$.  In fact these two cases are really symmetric, but we will consider each of them in the interest of completeness.\\
\underline{Case I}: $f\phi = p_{1}$ and $g\phi = p_{2}$.  Since $f\phi(z) = p_{1}(z) = 1$, $z\nleq x$.  Since $g\phi(z) = p_{2}(z) = 0$, $z\leq y$.\\
\underline{Case II}: $f\phi = p_{2}$ and $g\phi = p_{1}$.  Since $f\phi(z) = p_{2}(z) = 0$, $z\leq x$.  Since $g\phi(z) = p_{1}(z) = 1$, $z\nleq y$.\\

Now, we have\\
\centerline{
\xymatrix{
L\ar@<.5ex>[r]^{p_{1}}\ar@<-.5ex>[r]_{p_{2}} & T
}
}
with $p_{1}\ne p_{2}$.  In particular, $p_{1}(z)\ne p_{2}(z)$. Let $\alpha:S\to L$ be defined by $0\mapsto 0$, $a\mapsto z$ and $1\mapsto 1$.  Since the cardinality of $L$ is strictly greater than the cardinality of $S$, $\alpha$ cannot be bijective.  Thus, $\alpha\circ\phi\ne 1_{L}$, but $\alpha\circ\phi(z) = z$.\\ If we let $\beta = t\circ p_{1}$ and $\gamma = t\circ p_{2}$ where $t:T\to L$ is the canonical arrow, then we have four distinct arrows in $\mbox{Arr}\langle L,L\rangle$, namely $\alpha\circ\phi$, $\beta$ , $\gamma$ and $1_{L}$.  
\end{proof}

\section{$S$ is Preserved Up to Isomorphism by Equivalences}\label{eqpresS}
As we have seen, $S$ is the only locale which possesses two points\ref{Shas2pts}, three automorphisms\ref{Sis2pt3autogen} and is a generator in the category of frames\ref{Sis2pt3autogen}.  Since each of these are preserved by functorial equivalances, so is the isomorphism class of $S$.

\section{The Structure of $\mbox{Arr}\langle L,S\rangle$}
In this section we develop a characterization of the partial order on the sets $\mbox{Arr}\langle K,L\rangle$ in the category of frames using only concepts preserved by equivalences.  To do this we will need some facts regarding coproducts in that category.  Elements of a sum $A+B$ are joins of elements of the form $a\otimes b$ as $a$ ranges through $A$ and $b$ through $B$. These satisfy certain relations which are proved in section 1.4 of \cite{Borceux}:
\begin{align}
\vee_{t\in T}\left(a_{t}\otimes b\right) &= \left(\vee_{t\in T}a_{t}\right)\otimes b\\\label{genprop1}
\vee_{t\in T}\left(a\otimes b_{t}\right) &= a \otimes\left(\vee_{t\in T}b_{t}\right)\\
\left(a_{1}\wedge a_{2}\right)\otimes\left(b_{1}\wedge b_{2}\right) &= \left(a_{1}\otimes b_{1}\right) \wedge \left(a_{2}\otimes b_{2}\right)\label{genprop3}
\end{align}
Here we prove that $\mbox{Arr}\langle S,L\rangle\simeq \mbox{Arr}\langle T(S),T(L)\rangle$ as frames for any endo-equivalence $\mathbbm{T}$ on the category of frames.  \\
We will also need the following fact, also from Borceux III 1.4:
\begin{lemma}\label{injofsums}
Let $A$ and $B$ be frames.  Then, the injections $i:A\to A+B$ and $j:B\to A+B$ are defined by the rules $a\mapsto a\otimes 1$ and $b\mapsto 1\otimes b$ respectively.
\end{lemma}
We need the following fact which is proved in Borceux III 1.4.3 in the middle of page 23:
\begin{lemma}
Let $L$ be a frame and denote by $L+L$ the coproduct of $L$ with itself.  Then, the codiagonal arrow $\triangledown :L+L\to L$ is defined by 
$$\bigvee_{k\in K}b_{k}\otimes c_{k}\mapsto \bigvee_{k\in K}\left(b_{k}\wedge c_{k}\right)$$
\end{lemma}
\begin{lemma}\label{sumcomtens}
Let \\
\centerline{
\xymatrix{
S\ar@<.5ex>[r]^{f}\ar@<-.5ex>[r]_{g} & L
}
}
Let $f+g$ be the sum of the arrows $f$ and $g$, that is, $f+g$ is the unique arrow which makes\\\label{diagsumcomtens}
\centerline{
\xymatrix{
S\ar[dr]^{i_{S}}\ar[rrr]^{f} &&& L\ar[dl]_{i_{L}}\\
& S\ar[r]^{f+g} & L+L\\
S\ar[ur]^{j_{S}}\ar[rrr]_{g} &&& L\ar[ul]^{j_{L}}
}
}
commute.  Let $\alpha$ and $\beta$ be arbitrary elements of $S$.  Then, 
$$f+g(\alpha\otimes\beta) = f(\alpha)\otimes g(\beta)$$
\end{lemma}
\begin{proof}
We have
\begin{align}
f+g(\alpha\otimes\beta )&= f+g\left([\alpha\otimes 1]\wedge [1\otimes \beta]\right)\\\label{sumcomtens1}
&=f+g\left([\alpha\otimes 1]\right)\wedge f+g\left([1\otimes \beta]\right)\\\label{sumcomtens2}
&=i_{L}\circ f(\alpha)\wedge j_{L}\circ g(\beta)\\\label{sumcomtens3}
&=\left(f(\alpha)\otimes 1\right)\wedge \left(1\otimes g(\beta)\right)\\\label{sumcomtens4}\\ 
&=f(\alpha)\otimes g(\beta)\label{sumcomtens5}
\end{align}
where \ref{sumcomtens1} follows from property \ref{genprop3} of elements of $A+A$,  \ref{sumcomtens2} follows from the fact that $f+g$ is an arrow in the category of frames and so preserves the meet of two elements, \ref{sumcomtens3} follows by commutativity of \ref{diagsumcomtens}, \ref{sumcomtens4} follows by the definition of the sums in the category of frames (\ref{injofsums}), and \ref{sumcomtens5} follows by property \ref{genprop3} of elements of the sum of frames.
\end{proof}
\begin{theorem}\label{catorder}
Suppose \\
\centerline{
\xymatrix{
S\ar@<.5ex>[r]^{f}\ar@<-.5ex>[r]_{g} & L
}
}
are arrows in the category of frames with $S$ the Sierpinski Locale.  Suppose $f\leq g$ in the pointwise partial ordering of $\mbox{Arr}\langle S,L\rangle$. Then, 
$$\triangledown\circ\left(f+g\right) = f$$  
\end{theorem}
\begin{proof}
Suppose $f\leq g$.  Then, $f(a)\leq g(a)$ where $a$ is the non-zero element of $S$.  Thus, we have
$$\triangledown_{L}\left(f+g\right)(a) = f(a)\wedge g(a) = f(a)$$
Since, frame arrows must take $1\mapsto 1$ and $0\mapsto 0$, we have that
$$\triangledown\circ\left(f+g\right) = f$$  
\end{proof}
\begin{theorem}\label{eqpresorder}
Let $\mathbbm{T}$ be an equivalence on the category of frames.  Let $L$ and $K$ be frames. Then, $\mathbbm{T}:\mbox{Arr}\langle L,K\rangle\to\mbox{Arr}\langle \mathbbm{T}L,\mathbbm{T}K\rangle$ preserves order. 
\end{theorem}
\begin{proof}
Let $f$ and $g$ be arbitrary elements of $\mbox{Arr}\langle L,K\rangle$.  Suppose $f\leq g$.  Then, by \ref{catorder}, we have that 
$$f = \triangledown\circ\left(f+g\right)$$  
Thus,
\begin{align}
\mathbbm{T}f &= \mathbbm{T}\left(\triangledown_{K}\circ\left(f+g\right)\right)\\
\Rightarrow \\
&= \mathbbm{T}\triangledown_{K}\circ\mathbbm{T}\left(f+g\right)\\
&=\triangledown_{\mathbbm{T}K}\circ\left(\mathbbm{T}f+\mathbbm{T}g\right)\\
\end{align}
Thus, $\mathbbm{T}f\leq\mathbbm{T}g$.
\end{proof}
\begin{lemma}\label{eqbijpresord}
Let $\mathbbm{T}$ be an equivalence on the category of frames.  Let $L$ be a frame. Then, $\mathbbm{T}$ induces an order preserving function $\mbox{Arr}\langle S,L\rangle\to \mbox{Arr}\langle \mathbbm{T}S,\mathbbm{T}L\rangle$ which possesses an order preserving inverse function.
\end{lemma}
\begin{proof}
Suppose $\mathbbm{R}$ is the left adjoint of the pair given in Definition \ref{defeq}. Let $f:S\to L$ be arbitrary.  Then, by \ref{equivequiv}, we have the commutative square\\
\centerline{\label{epeqdiag}
\xymatrix{
S\ar[r]^{\epsilon_{S}}\ar[d]_{f} & \mathbbm{R}\mathbbm{T}S\ar[d]^{\mathbbm{R}\mathbbm{T}f}\\
L\ar[r]_{\epsilon_{L}} & \mathbbm{R}\mathbbm{T}L
}
}
As we have seen, though(\ref{eqpresS}, equivalences on the category of frames carry $S$ to itself (up to trivial re-labeling of elements) so we may take $\epsilon_{S}$ to be the identity on $S$.  Define $\mathbf{G}:\mbox{Arr}\langle \mathbbm{T}S,\mathbbm{T}L\rangle\to\mbox{Arr}\langle S,L\rangle$ by the assignment $\mathbbm{T}f\mapsto \epsilon^{-1}\mathbbm{R}\mathbbm{T}f$.  By commutativity of \ref{epeqdiag}, we get that $\epsilon^{-1}\circ\mathbbm{R}\mathbbm{T}f = f$. Thus, $\mathbf{G}\mathbbm{T}$ is the identity on $\mbox{Arr}\langle S,L\rangle$.  Let $g\in\mbox{Arr}\langle\mathbbm{T}S,\mathbbm{T}L\rangle$ be arbitrary.  Then, since $\mathbbm{T}$ defines a bijection on sets of arrows, we can find some $h\in\mbox{Arr}\langle S,L\rangle$ so that $\mathbbm{T}f = h$.  Thus, $\mathbbm{T}\mathbf{G}g = \mathbbm{T}\mathbf{G}\mathbbm{T}h = \mathbbm{T}h = g$. Thus, $\mathbbm{T}\mathbf{G}$ is the identity on $\mbox{Arr}\langle\mathbbm{T}S,\mathbbm{T}L\rangle$.\\
Since $\mathbf{G}$ is defined to by a composition of order-preserving functions, it, too, is order preserving. 
\end{proof}
\begin{theorem}
Let $L$ be a locale.  Let, as usual, $S$ be the Sierpinski locale.  Let $\mathbbm{T}$ be an equivalence on the category of frames and $\mathbbm{R}$ its right adjoint.  Then, $\mathbbm{T}$ induces a frame isomorphism from $\mbox{Arr}\langle S,L\rangle $ to $\mbox{Arr}\langle\mathbbm{T}S,\mathbbm{T}L\rangle$. 
\end{theorem}
\begin{proof}
Let $\mathbbm{T}$ be an equivalence on the category of frames. Let $L$ be a frame.  Then, by Lemma\ref{eqbijpresord}, $\mathbbm{T}$ induces an order-preserving function from $\mbox{Arr}\langle S,L\rangle $ to $\mbox{Arr}\langle\mathbbm{T}S,\mathbbm{T}L\rangle$.  By Lemma\ref{partordisomimplisframeisom} this function is an isomorphism of frames.
\end{proof}
\section{The Category of Frames is Rigid}
Let $\mathbbm{T}$ be an equivalence on the category of frames, $S$ be the Sierpinski Locale, and $L$ any frame at all.  Then, we have seen that
\begin{align*}
L\simeq& \mbox{Arr}\langle S,L\rangle\\
\simeq& \mbox{Arr}\langle S,\mathbbm{R}\mathbbm{T}(L)\rangle\\
\simeq& \mbox{Arr}\langle \mathbbm{T}S,\mathbbm{T}L\rangle\\
\simeq& \mbox{Arr}\langle S,\mathbbm{T}L\rangle\\
\simeq& TL
\end{align*}
Thus, $\mathbbm{T}$ is an essential isomorphism, and, so, by definition, the category of frames is rigid.

\section{The Category of Locales is Rigid}
 If we take the definition of the category of locales to be that it is the opposite of the category of locales, then it is not difficult to see that it is rigid.  Indeed, if any category $\mathcal{C}$ is rigid then so is its opposite category $\mathcal{C}^{op}$.  We recall the following definition
\begin{definition}
Let $\mathcal{C}$ be a category and $\mathbbm{T}:\mathcal{C}\to\mathcal{C}$ an endofunctor on $\mathcal{C}$.  Then, we define $\mathbbm{T}^{op}:\mathcal{C}^{op}\to\mathcal{C}^{op}$ as the functor which assigns to any object $A$ the object $\mathbbm{T}A$ and to any arrow $f^{op}:B\to A$ the arrow $\left(\mathbbm{T}f\right)^{op}$.
\end{definition}
The following result is hardly surprising
\begin{lemma}
Suppose $\mathbbm{T}$ is an endo-equivalence on $\mathcal{C}$.  Then so too is $\mathbbm{T}^{op}$ on $\mathcal{C}^{op}$.
\end{lemma}
\begin{proof}
Since $\mathbbm{T}$ is an endo-equivalence, it is full and faithful by \ref{equivequiv}.  Suppose $f^{op}:B\to A$ and $g^{op}:B\to A$ are distinct arrows in $\mathcal{C}^{op}$. Then, since $\mathbbm{T}$ is faithful so too are the arrows $\mathbbm{T}f$ and $\mathbbm{T}g$.  Thus, so too are the arrows  $\left(\mathbbm{T}f\right)^{op}$ and $\left(\mathbbm{T}g\right)^{op}$. Thus $\mathbbm{T}$ is faithful.  Let $g^{op}:\mathbbm{T}^{op}B\to \mathbbm{T}^{op}A$ be an arbitrary arrow in $\mathcal{C}^{op}$.  Then, since $\mathbbm{T}$ is faithful, there is some arrow $f:A\to B$ so that $\mathbbm{T}f = g$.  Thus, $\mathbbm{T}^{op}f^{op} = g^{op}$.\\
Thus, $\mathbbm{T}^{op}$ is full and faithful.  Thus $\mathbbm{T}^{op}$ is an equivalence by \ref{equivequiv}.
\end{proof}
\begin{theorem}
Let $\mathbbm{T}$ is an endo-equivalence isomorphic to the identity functor on $\mathcal{C}$, then $\mathbbm{T}^{op}$ is an endo-equivalence isomorphic to the identity functor on $\mathcal{C}^{op}$.
\end{theorem}
\begin{proof}
Suppose $\mathbbm{T}$ is an endo-equivalence isomorphic to the identity functor on $\mathcal{C}$.  Let $f:A\to B$ be an arbitrary arrow in $\mathcal{C}$. Then, we have the commutative diagram\\
\centerline{
\xymatrix{
A\ar[r]^{\alpha_{A}}\ar[d]_{f} & \mathbbm{T}A\ar[d]^{\mathbbm{T}f}\\
B\ar[r]_{\alpha_{B}} & \mathbbm{T}B
}
}
where $\alpha_{*}$ are isomorphisms.  Thus, by definition, we have the commutative diagram\\
\centerline{
\xymatrix{
A & \mathbbm{T}A\ar[l]^{\alpha_{A}^{op}}\\
B\ar[u]^{f^{op}} & \mathbbm{T}B\ar[l]^{\alpha_{B}^{op}}\ar[u]_{\mathbbm{T}f^{op}}
}
}
in $\mathcal{C}^{op}$.  Since the dual notion of isomorphism is isomorphism, the $\alpha_{*}^{op}$ are also isomorphisms.  Thus, $\mathbbm{T}^{op}$ is naturally isomorphic to the identity on $\mathcal{C}^{op}$ 
\end{proof}
\begin{theorem}
The category of locales, defined to be the opposite category of the category of frames is rigid.
\end{theorem}

\bibliographystyle{plain}
\bibliography{tau}

\end{document}